\documentclass{amsart}
 \usepackage{graphicx, amsmath, amssymb, amsthm, amsfonts, fullpage, latexsym,enumerate, xcolor, commath, tikz, tikz-cd, verbatim, hyperref,standalone, mathtools, booktabs, hyperref}
 \usepackage[shortlabels]{enumitem}
\usepackage[backend=biber,
style=alphabetic,
doi=false,
isbn=false,
url=false,
eprint=false]{biblatex}
\makeatletter
\newtheorem*{rep@theorem}{\rep@title}
\newcommand{\newreptheorem}[2]{%
\newenvironment{rep#1}[1]{%
 \def\rep@title{#2 \hyperref[##1]{\ref*{thrm:main}.\ref*{##1}}}%
 \begin{rep@theorem}}%
 {\end{rep@theorem}}}
\makeatother

\newreptheorem{theorem}{Theorem}

\makeatletter
\newcommand{\newrepdoubletheorem}[2]{%
\newenvironment{repdouble#1}[1]{%
 \def\rep@title{#2 \hyperref[##1]{\ref*{thrm:main}.iii.\ref*{##1}}}%
 \begin{rep@theorem}}%
 {\end{rep@theorem}}}
\makeatother

\newrepdoubletheorem{theorem}{Theorem}

\makeatletter
\newcommand{\newrepintrotheorem}[2]{%
\newenvironment{repintro#1}[1]{%
 \def\rep@title{#2 \hyperref[##1]{\ref*{##1}}}%
 \begin{rep@theorem}}%
 {\end{rep@theorem}}}
\makeatother

\newrepintrotheorem{theorem}{Theorem}

\title{Shifted and threshold matroids}

\author{Ethan Partida}
\address{Department of Mathematics, Brown University, Box 1917, Providence, RI 02912}
\email{ethan\_partida@brown.edu}

\theoremstyle{definition}
\newtheorem{defn}{Definition}[section]
\newtheorem{remark}[defn]{Remark}
\newtheorem{example}[defn]{Example}
\theoremstyle{plain}
\newtheorem{lemma}[defn]{Lemma}
\newtheorem{prop}[defn]{Proposition}
\newtheorem{corollary}[defn]{Corollary}
\newtheorem{exthrm}[defn]{Theorem/Counterexample}

\newtheorem{thrm}[defn]{Theorem}

\newtheorem{question}[defn]{Question}
\newtheorem{mainthrm}{Theorem}

\newcommand{\R}{\mathbb{R}}

\begin{document}
\begin{abstract}
We characterize the class of threshold matroids by the structure of their defining bases. We also give an example of a shifted matroid which is not threshold, answering a question of Deza and Onn. We conclude by exploring consequences of our characterization of threshold matroids: We give a formula for the number of isomorphism classes of threshold matroids on a ground set of size $n$. This enumeration shows that almost all shifted matroids are not threshold. We also present a polynomial-time algorithm to check if a matroid is threshold and provide alternative and simplified proofs of some of the main results of Deza and Onn.
\end{abstract}
\maketitle
\section{Introduction}
Let $T$ be a non-empty subset of $[n]$ with size $k$. The size $k$ subsets of $[n]$ which are less than or equal to $T$ under the component-wise partial order form the bases of a matroid $M=\langle T \rangle$. The matroids $M=\langle T \rangle$ arising from this construction are known as \emph{shifted matroids}. As the choice of subset $T\subseteq [n]$ completely determines the structure of $M= \langle T \rangle$, we refer to $T$ as the \emph{defining basis} of $M$. Shifted matroids are an important and well-studied class of matroids which appear under many different names and in various contexts. A non-exhaustive list of such names is: nested matroids \cite{Hampe}, non-exchangeable matroids \cite{DO}, generalized Catalan matroids \cite{BONIN} and Schubert matroids \cite{Sohoni}. Relevant to the present work, exchangeable matroids were recently introduced by Deza and Onn \cite{DO} in their work on $k$-hypergraphs. The class of exchangeable matroids is exactly the class of matroids which are \emph{not} shifted matroids. Our work was motivated by, and answers, various questions posed by Deza and Onn; see Theorem \ref{thrm:2-trade} and Counterexample \ref{ex:counter}.

Threshold matroids are the class of matroids whose bases can be separated from non-bases by a weight function on their ground set. The description of threshold matroids and, more generally, threshold hypergraphs using weight functions makes them interesting objects from the perspective of combinatorial optimization and polytope theory; see \cite[Section $9$]{Boolean}, \cite{GK}, \cite[Section $2$]{BS} and \cite[Remark $7.8$]{CLLPPS}.

It can be shown that all threshold matroids are shifted; see Lemma \ref{lem:shift->thresh}. In recent work, Deza and Onn ask whether the converse is true: Are all shifted matroids threshold \cite[Question $7.2$]{DO}? As evidence for their question, they show that if a shifted matroid is also paving, or binary, or rank three, then it is threshold. Despite these results, we show that the question is false in the strongest possible manner.

\begin{mainthrm}
  \label{thrm:almost}
  Almost all shifted matroids are \underline{not} threshold. That is,
    \[\lim_{n\to\infty}\frac{\text{$\#$ of isomorphism classes of threshold matroids on a ground set of size $n$}}{\text{$\#$ of isomorphism classes of shifted matroids on a ground set of size $n$}}=0. \]
\end{mainthrm}

The smallest shifted matroids that are not threshold are of rank $4$ and are defined on a ground set of size $8$. For example, the shifted matroid $M= \langle 2468 \rangle$ defined by the subset $2468 \subseteq [8]$ is not threshold. On a ground set of size $14$, $50.00$ percent of all shifted matroids are threshold. On a ground set of size $35$, less than $0.01$ percent of shifted matroids are threshold. For an elementary proof that the matroid $M=\langle 2468\rangle$ is not threshold, see Counterexample \ref{ex:counter}.

We prove Theorem \ref{thrm:almost} by first characterizing threshold matroids in terms of their defining bases (Theorem \ref{thrm:main}) and then enumerating the set of isomorphism classes of threshold matroids on a ground set of size $n$ (Theorem \ref{thrm:enumeration}). Our enumeration shows that the number of isomorphism classes of threshold matroids on a ground set of size $n$ is $O(n^6)$. The number of isomorphism classes of shifted matroids on a ground set of size $n$ is $2^n-1$. Comparing these enumerations leads to a proof of Theorem \ref{thrm:almost}.

\begin{mainthrm}
  \label{thrm:enumeration}
Let $F(n)$ be the number of isomorphism classes of threshold matroids on a ground set of size $n$. Then
\[F(n) = \binom{n+1}{2} + \binom{n+1}{4}+ \binom{n+1}{6} + \binom{n-1}{6}\]
where we take $\binom{n}{k}=0$ whenever $n<k$.
\end{mainthrm}

While Deza and Onn were the first to ask whether all shifted matroids are threshold, generalizations of this question to $k$-hypergraphs have been posed, and answered in the negative, numerous times. See, for example,  \cite[Theorem $3.1$]{KR}, \cite[Example $2.2$]{RRST}, \cite[Theorem $7.3.3$]{Muroga} and \cite[Example $2.15$]{BS}. 

Our characterization of threshold matroids relies on the structure of the defining bases of shifted matroids. Given a subset $T\subseteq [n]$, we will refer to the \emph{gaps} and \emph{blocks} of $T$. We give formal definitions of these terms in Definition \ref{defn:blocks}, but for now, we illustrate the concepts with an example. The subset $23479\subseteq [10]$ has three blocks and four gaps. The first block is $234$, the second block is $7$, and the third block is $9$ while $1$, $56$, $8$ and $10$ are the first, second, third and fourth gaps. 
\begin{mainthrm}
  \label{thrm:main}
Suppose $M$ is a shifted matroid and let $T$ be the defining basis of the shifted matroid obtained from contracting all coloops of $M$.
Then:
\begin{enumerate}[i)]
\item If $T$ has four or more blocks, $M$ is not threshold.
\item If $T$ has two or fewer blocks, $M$ is threshold.
\item If $T$ has three blocks:
\begin{enumerate}[a)]
\item If the second block or gap of $T$ has size one, $M$ is threshold.
\item Otherwise, $M$ is not threshold. 
\end{enumerate}
\end{enumerate}
\end{mainthrm}

For the cases in which $M$ is threshold, we explicitly construct a weight function as a certificate of thresholdness. To show the cases of non-thresholdness, we make use of a tool which we call \emph{uniform asummability}. This tool is the uniform analogue of \emph{asummability} from the simple games literature \cite[Section $2.6$]{TZ} and the Boolean functions literature \cite[Chapter $9$]{Boolean}. Along the way, we also prove the following theorem and answer another question of Deza and Onn \cite[Question $1.1$ of][for matroids]{DO}.

\begin{mainthrm}
  \label{thrm:2-trade}
A matroid is threshold if and only if it is $2$-uniform asummable. Equivalently, a rank $k$ matroid $M$ is threshold if and only if for all bases $B_1,B_2$ and size $k$ non-bases $D_1,D_2$, we have 
\[B_1 \cup B_2 \neq D_1 \cup D_2 \text{ or } B_1\cap B_2 \neq D_1 \cap D_2. \]
\end{mainthrm}

Lastly, we explore other consequences of our characterization of threshold matroids. Along with a proof of Theorem \ref{thrm:enumeration}, we provide alternative and simplified proofs of some of the main results of \cite{DO} (Theorem \ref{thrm:DO}) and produce an algorithm, polynomial in the number of bases, to check if a matroid is threshold (Theorem \ref{thrm:polynomial}).

Much of our work references the paper \cite{DO} by Deza and Onn. Their work begins in the more general setting of $k$-hypergraphs and then later specializes to $k$-hypergraphs that form the bases of a matroid. As our paper is entirely focused on matroids, some of our terminology is different than the terminology in \cite{DO}. In Table \ref{tab:table}, we provide a translation between the terminology used in this paper and the terminology used in \cite{DO}.

\begin{table}[ht] 
\label{tab:table}
\centering
    \begin{tabular}{cc}\toprule
         Our terminology & Terminology in \cite{DO} \\ 
         \midrule
         Threshold (Definition \ref{def:thresh}) & Separable \cite[Section $1$]{DO}\\[1em]
         Not threshold (Definition \ref{def:thresh}) & Equatable \cite[Section $1$]{DO}\\[1em]
         Shifted (Definition \ref{def:shifted}) & Not exchangeable \cite[Section $1$]{DO}\\[1em]
         $2$-uniform asummable (Definition \ref{def:trade_robust}) & Unnamed but the subject of \cite[Question $1.1$]{DO}\\[1em]
    \end{tabular}
    \caption{Translation between the terminology used in this paper and the terminology used in \cite{DO}. References for the equivalences, from top to bottom, are: Remark \ref{rem:thresh}, \cite[Lemma $2.1$]{DO}, Lemma \ref{lem:shifted_char} and Lemma \ref{lem:ele}.}
\end{table}

The structure of the paper is as follows. In Section \ref{sec:back}, we introduce shifted and threshold matroids. In Section \ref{sec:counter}, we outline the relationship between shifted and threshold matroids and give the first example of a shifted matroid that is not threshold. In Section \ref{sec:trade}, we develop the tool of uniform asummability to aid in our characterization of threshold matroids. In Section \ref{sec:main}, we prove our main theorem, a characterization of threshold matroids. Finally, in Section \ref{sec:consequences}, we prove various consequences of our main theorem.

\section*{Acknowledgments}
I am grateful to C. Klivans for suggesting this problem and for her support throughout. I would also like to thank S. Onn for enlightening discussions early in this project, M. Chan for their feedback on an early draft of this paper and C. Chenevière for suggesting a cleaner proof of Theorem \ref{thrm:enumeration}. Finally, I sincerely thank the anonymous reviewers for their detailed and incisive comments.

The author was partially supported by the National Science Foundation through FRG
DMS-2053221 and the Department of Education through a GAANN fellowship.

\section{Background}
\label{sec:back}

This paper studies two special kinds of matroids, shifted matroids and threshold matroids. Both of these classes of matroids are most naturally defined using the \emph{bases} of a matroid. We recall the bases definition of a matroid below. Throughout the paper, we will often refer to the basic concepts of matroid theory (such as circuits, independent sets and deletion/contraction) without definition. These can be found in any standard reference on matroid theory, such as \cite{oxley}.

\begin{defn}\label{defn:matroid}
A rank $k$ matroid $M$ on ground set $E$ is a non-empty collection of $k$-element subsets of $E$, called \emph{bases}, which satisfy the basis-exchange axiom. The basis-exchange axiom says that for every $B_1,B_2\in M$ and $b_1\in B_1\setminus B_2$ there exists $b_2\in B_2\setminus B_1$ such that $(B_1\setminus b_1)\cup b_2 \in M$.
\end{defn}

\subsection{Shifted matroids}
Consider the set $[n]$ with its total ordering $1<2<\ldots <n$. Let $W(n)$ be the set of weakly increasing words whose alphabet is $[n]$ and let $W(k,n)$ be the set of weakly increasing words of length $k$ whose alphabet is $[n]$. Formally, a weakly increasing word $T\in W(k,n)$ is a function $T:[k]\to [n]$ which is weakly monotonically increasing, i.e., $T(1) \leq T(2) \leq \ldots \leq T(k)$. Throughout this paper, we often identify the word $T$ with the string $t_1 t_2 \cdots t_k$ where $t_i= T(i)$.

\begin{defn}\label{defn:cmpt}
  The \emph{component-wise partial order} $\preceq$ on $W(k,n)$ declares $S\preceq T$ if $S(i)\leq T(i)$ for all $1\leq i \leq k$.
\end{defn}
\noindent See Figure \ref{fig:hasse_words} for the Hasse diagram of $(W(2,3), \preceq)$.

\begin{figure}[ht]
\centering{
\begin{tikzpicture}
  \node (11) at (0,0) {$11$};
  \node (12) at (0,1) {$12$};
  \node (13) at (1,2) {$13$};
  \node (22) at (-1,2) {$22$};
  \node (23) at (0,3) {$23$};  
  \node (33) at (0,4) {$33$};
  \draw (11) to (12);
  \draw (12) to (13);
  \draw (12) to (22);
  \draw (13) to (23);
  \draw (22) to (23);
  \draw (23) to (33);
\end{tikzpicture}
}
\caption{The component-wise partial order of $W(2,3)$.}
\label{fig:hasse_words}
\end{figure}
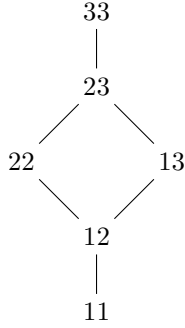

The following lemma shows that the component-wise partial order is invariant under permutation of words. That is, if we can rearrange a word $T$ so that it is larger than $S$ at every index, then $T$ is still larger than $S$ at every index after we re-sort $T$ into a weakly increasing word.

\begin{lemma}
  \label{lem:weaker_prec}
Let $S,T\in W(k,n)$. Then $S\preceq T$ if and only if there is a bijection $\pi:[k]\to [k]$ such that $S(i)\leq T\circ \pi (i)$ for all $i\in[k]$.
\end{lemma}
\begin{proof}
  Suppose $S\preceq T$ and let $\pi$ be the identity permutation. By definition, $S(i)\leq T\circ \pi(i)= T(i)$ for all $i\in[k]$ and the forward direction follows. We now focus on the backward direction.

  Suppose there exists a permutation $\pi:[k]\to [k]$ such that $S(i)\leq T\circ \pi(i)$ for all $i\in[k]$. If $\pi(i)< \pi(i+1)$ for all $i\in [k]$, then $\pi$ is the identity permutation and we are done. Instead, assume there is some index $i_0$ where $\pi(i_0)> \pi(i_0+1)$. Let $\tilde \pi= (i_0\,\, i_0+1)\pi$ be the permutation composing $\pi$ with the adjacent transposition $(i_0\,\, i_0+1)$. We show that $S(i) \leq T \circ \tilde \pi(i)$ for all $i\in[k]$. Since we composed $\pi$ with an adjacent transposition, we only need to check this inequality in two indices, $i_0$ and $i_0+1$. We can check this using the fact that $T$ and $S$ are weakly monotonically increasing. Observe that
  \[T\circ \tilde \pi(i_0) = T\circ\pi(i_0+1)\ge S(i_0+1) \ge S(i_0)   \]
  and
  \[T\circ \tilde \pi(i_{0}+1) = T\circ\pi(i_0) \geq T\circ \pi(i_0+1) \geq S(i_0+1). \]
 Since $S_n$ is generated by adjacent transpositions, we can repeatedly do this modification until $\tilde \pi$ is the identity permutation.
\end{proof}

In terms of strings of letters, Lemma \ref{lem:weaker_prec} states that $S\preceq T$ if there is a one to one assignment $\psi$ between the letters of $T=t_1\cdots t_k$ and the letters of $S=s_1\cdots s_k$ such that $t_i \geq \psi(t_i)$.

The set of size $k$ subsets of $[n]$, denoted $\binom{[n]}{k}$, can be identified with the subset of $W(k,n)$ of strictly increasing words. Hence $\binom{[n]}{k}$ inherits a partial ordering from the component-wise partial ordering of $W(k,n)$. We also call this ordering of $\binom{[n]}{k}$ the component-wise partial order. For the rest of this paper, we will identify $\binom{[n]}{k}$ with its image in $W(k,n)$ and refer to sets by their associated strictly increasing word. We can use the component-wise partial ordering of $\binom{[n]}{k}$ to define shifted matroids.

\begin{defn}\label{def:shifted}
A \emph{shifted matroid} $M$ of rank $k$ is one which is isomorphic to a matroid whose bases form an order ideal in $\binom{[n]}{k}$ under the component-wise partial order. 
\end{defn}

A theorem of Klivans \cite{K-thesis} greatly simplifies the conceptual complexity of shifted matroids.

\begin{thrm}[{\cite[Theorem $5.4.1$]{K-thesis}}]
An order ideal in the component-wise partial order of $\binom{[n]}{k}$ is a matroid if and only if it is principal. Thus, any rank $k$ shifted matroid is isomorphic to the order ideal generated by some size $k$ subset $T\in \binom{[n]}{k}$.
\end{thrm}

We will usually assume our shifted matroids $M$ are defined on the ground set $[n]$ and their bases form a principal order ideal under the component-wise partial order. In this case, we write $M=\langle T \rangle$ to indicate that $M$ has bases equal to the principal order ideal generated by $T\in \binom{[n]}{k}$. We call $T$ the \emph{defining basis} of $M$. See Figure \ref{fig:shifted} for an example of the shifted matroid with defining basis $24\subseteq \binom{[4]}{2}$.

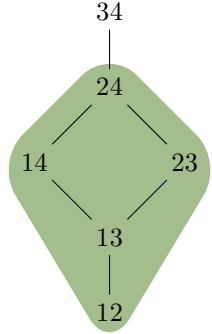
\begin{figure}[ht]
\centering{
\begin{tikzpicture}

\filldraw[rounded corners=0.5cm, thick, green, fill opacity = 0.25]  (0,-0.5) -- (-1.5,2) -- (0,3.5) -- (1.5,2) -- cycle;

\node (12) at (0,0) {12};
\node (34) at (0,4) {34};
\node (24) at (0,3) {24};
\node (14) at (-1,2) {14};
\node (23) at (1,2) {23};
\node (13) at (0,1) {13};
\draw (34) -- (24);
\draw (24) -- (14);
\draw (24) -- (23);
\draw (13) -- (14);
\draw (13) -- (23);
\draw (13) -- (12);
\end{tikzpicture}
}
\caption{The shifted matroid with defining basis $24\subseteq [4]$ has bases $24,14,23,13$ and $12$.}
\label{fig:shifted}
\end{figure}

We now introduce the \emph{vicinal preorder} of a matroid. This concept will let us give an intrinsic characterization of shifted matroids (Lemma \ref{lem:shifted_char}) and will be used to construct an algorithm to test if a given matroid is threshold (Theorem \ref{thrm:polynomial}).

\begin{defn}
Given a rank $k$ matroid $M$ on a ground set $E$ and $i\in E$, the \emph{open and closed neighborhoods} of $i$ in $M$ are the following two subcollections of $\binom{E}{k-1}$:
\begin{align*}
N_M(i) &\coloneqq \left\{ B\setminus i : \text{ $i\in B$ and $B$ a basis of $M$}\right\}\\
N_M[i] &\coloneqq \left\{ B\setminus j : \text{ $i\in B$, $j\in B$ and $B$ a basis of $M$}\right\}.
\end{align*}

Note that $N_{M}(i)\subseteq N_{M}[i]$. The \emph{vicinal preorder} of $M$ on $E$ declares $i\preceq_M j$ whenever $N_M[i] \supseteq N_M(j)$.
\end{defn}

\begin{thrm}[{\cite[Theorem $1$]{Klivans}}]
  \label{thrm:vicinal}
The vicinal preorder of $M$ on $E$ is total if and only if $M$ is shifted. If $M$ is shifted, any bijection between $E$ and $[n]$ which preserves the vicinal preorder induces an isomorphism between $M$ and a shifted matroid $\tilde M= \left\langle T \right\rangle$ on ground set $[n]$.
\end{thrm}

Using Theorem \ref{thrm:vicinal},
we give an intrinsic characterization of all shifted matroids in terms of their bases. This characterization is important because it shows that the class of shifted matroids is the same as the class of non-exchangeable matroids; this class of matroids was recently studied in \cite{DO}. See Table \ref{tab:table} for a dictionary between the terminology used in this paper and the terminology used in \cite{DO}.
\begin{lemma}
  \label{lem:shifted_char}
A matroid $M$ on ground set $E$ is shifted if and only if for all bases $B_1,B_2\in M$ and $b_1\in B_1\setminus B_2$, $b_2\in B_2\setminus B_1$, either $(B_1\setminus b_1)\cup b_2 \in M$ or $ (B_2\setminus b_2)\cup b_1 \in M$.
\end{lemma}
\begin{proof}
  If $M$ is a shifted matroid then there is a bijection $\varphi:E\to[n]$ which maps the bases of $M$ to an order ideal in the component-wise partial order. Thus for any two bases $B_1, B_2$ and $b_1\in B_1\setminus B_2$, $b_2\in B_2\setminus B_1$ either $\varphi(b_1)<\varphi(b_2)$ or $\varphi(b_2)<\varphi(b_1)$. However, since the bases of $M$ form an order ideal under $\varphi$, this means that either $(B_1\setminus b_1)\cup b_2$ or $(B_2\setminus b_2)\cup b_1$ is a basis of $M$.

  If $M$ is not a shifted matroid, then, by Theorem \ref{thrm:vicinal}, the vicinal preorder is not total. That is, there exist elements $i,j\in E$ and bases $B_i,B_j$ such that $i\in B_i$, $j\in B_j$ and if $B_i'$ and $B_j'$ are bases that contain $i$ and $j$ then $B_j'\not\supseteq (B_i\setminus i)$ and $B_i'\not\supseteq (B_j\setminus j)$. Note that this condition implies that $j\not\in B_i$ and $i\not\in B_j$. Therefore $i\in B_i\setminus B_j$, $j\in B_j\setminus B_i$ and neither $(B_i\setminus i)\cup j$ nor $(B_j\setminus j)\cup i$ are bases of $M$.
\end{proof}

Throughout this paper, we will focus on the structure of the defining basis of a shifted matroid. In particular, we will often refer to the gaps and blocks of a non-empty subset of $[n]$. We define these notions below.

\begin{defn}
  \label{defn:blocks}
Let $T$ be a non-empty subset of $[n]$. Then $T$ can be written uniquely as a concatenation of subwords \[T=T_1 T_2 \cdots  T_\ell\] which are the inclusion-maximal contiguous subsets of $T$. 
\begin{itemize}
\item We call $T_i$ the $i$th \emph{block} of $T$. 
\item We call the $j$th block of $[n]\setminus T$, the $j$th \emph{gap} of $T$.
\end{itemize}
\end{defn}

\begin{example}
The set $T= 2678\subseteq [9]$ has two inclusion-maximal contiguous subsets, namely $2$ and $678$. Therefore, $T$ has two blocks and they are equal to $T_1=2$ and $T_2=678$.

The complement of $T$ can be written as the concatenation of $1$, $345$ and $9$. As these sets are the inclusion-maximal contiguous subsets of the complement, they are the three gaps of $T$.
\end{example}

We close this subsection with a technical lemma characterizing the circuits of a shifted matroid in terms of its defining basis. Recall that a circuit of a matroid is an inclusion-minimal subset of the ground set that is not contained in any basis. This lemma will be crucial in the proof of our characterization of threshold matroids as well as the proof of Theorem \ref{thrm:DO}.

\begin{lemma}
  \label{lem:circuits}
  Let $M= \langle T \rangle$ be a shifted matroid and $C=c_0c_1\cdots c_m$ be a circuit of $M$. Let $T=T_1\cdots T_\ell$ be the block decomposition of $T$. If $C$ is a loop, then $c_0$ is strictly larger than the last element of $T_\ell$. Otherwise, there exists some index $1\leq i \leq \ell$ such that:
  \begin{itemize}
    \item $|C| = |T_{i}|+|T_{i+1}|+\ldots+|T_\ell| +1$.
    \item $c_1\cdots c_m\preceq T_{i}T_{i+1}\cdots T_\ell$.
    \item If $i\neq 1$, then $c_0$ is strictly larger than the last element of $T_{i-1}$.
  \end{itemize}
\end{lemma}
\begin{proof}
  If $C$ is a loop, then $m=0$ and $c_0$ is strictly larger than the last element of $T_\ell$. From here on out, assume that $C$ is not a loop. 

  For ease of indexing, suppose that $T=t_1\cdots t_k$ and $C=c_0c_1\cdots c_m$ with $1\leq m \leq k$. Since $M$ is shifted, all maximal proper subsets of $C$ are contained in a basis of $M$ if and only if $c_1\cdots c_m$ is contained in a basis of $M$. This occurs exactly when $c_1\cdots c_m \preceq t_{k-m+1}\cdots t_{k}$. In order for this condition to be met and $C$ to be dependent, we must have that $c_0 > t_{k-m}$ or $|C|=k+1$. If $|C|=k+1$, then $|C| = |T_1|+\ldots + |T_\ell|+1$ and $c_1\cdots c_m\preceq t_1\cdots t_m = T$.
  In this case, $C$ satisfies our conclusion.

  Now, assume that $|C|<k+1$ so that $c_0> t_{k-m}$. We have the inequalities
  \[t_{k-m}<c_0< c_1 \leq t_{k-m+1}. \]
  Hence, $t_{k-m}< t_{k-m+1}-1$ and $t_{k-m+1}$ is the start of some block of $T$. Let $T_{i}$ be the block for which $t_{k-m+1}$ is the start of. We now have that:
  \begin{itemize}
    \item $|C| = |T_{i}|+|T_{i+1}|+\ldots+|T_\ell| +1$.
    \item $c_1\cdots c_m\preceq t_{k-m+1}\cdots t_k= T_{i}T_{i+1}\cdots T_\ell$.
    \item $c_0$ is strictly larger than $t_{k-m}$, which is the last element of $T_{i-1}$.
  \end{itemize}
\end{proof}

\begin{example}
  We can use Lemma \ref{lem:circuits} to analyze the circuits of the shifted matroid $M=\langle 2468\rangle$ on ground set $[8]$. By the first bullet point, the circuits of $M$ have size $2$, $3$, $4$ and $5$. We can use the second and third bullet points of Lemma \ref{lem:circuits} to write down all such circuits:
  \begin{align*}
    \text{Size $2$ circuits: }& \{78\}\\
    \text{Size $3$ circuits: }& \{567, 568\}\\
    \text{Size $4$ circuits: }& \{3456,3457,3458,3467,3468\}\\
    \text{Size $5$ circuits: }& \{12abc \text{ where $3\leq a \leq 4$, $a< b \leq 6$ and $b< c \leq 8$}\}.
  \end{align*}
\end{example}

\subsection{Threshold matroids}
A threshold matroid is one whose bases are determined by a weight function, as follows:
\begin{defn}\label{def:thresh}
Let $M$ be a rank $k$ matroid on $E$. A matroid $M$ is a \emph{threshold matroid} if there exists a function $w:E\to \R$ such that, for all $B\in \binom{E}{k}$,
\[\text{$B$ is a basis of $M$} \iff w(B)=\sum_{b\in B} w(b) > 0.\]
When this is the case, we say $w$ is a weight function for $M$ and call $w(B)$ the \emph{weight} of the basis $B$.
\end{defn}
\begin{remark}\label{rem:thresh}
In recent work, Deza and Onn \cite{DO} introduce the notion of separable matroids. Their definition of a separable matroid is exactly the same as that of a threshold matroid, except that one only requires the weight of bases to be non-negative instead of positive. It can be shown, by perturbing the weight function slightly, that a matroid is threshold if and only if it is separable. We choose to use the name threshold because it coincides with that of threshold graphs and the earlier works of Bhanu Murthy--Srinivasan \cite{BS} and Klivans--Reiner \cite{KR}.
\end{remark}

\begin{remark}
  \label{rem:GK}
Previously, Giles and Kannan \cite{GK} defined and characterized a different notion of threshold matroid. For the rest of this remark, we will refer to their definition as \emph{non-uniform threshold}. A matroid $M$ is non-uniform threshold if there exists a weight function $w:E\to \R$ and a threshold $t\in \R$ such that a set $I\subseteq E$ is independent if and only if $w(I)\leq t$. Non-uniform thresholdness is a strictly stronger condition than thresholdness. For example, the shifted matroid $M= \langle 246 \rangle$ is threshold but not non-uniform threshold.  
\end{remark}

We close this subsection with two lemmas which show that thresholdness is preserved under duality and adding/contracting coloops. The first lemma is proven in \cite{DO} using different terminology than the current article. For a translation between the terminology used in \cite{DO} and the terminology used in this article, see Table \ref{tab:table}.

\begin{lemma}[{\cite[Lemma $2.3$]{DO}}]
\label{lem:thresh_dual}
Thresholdness is preserved by matroid duality. That is, $M$ is threshold if and only if $M^\perp$ is.
\end{lemma}

\begin{lemma}
\label{lem:thresh_contract}
Thresholdness is preserved by contracting coloops. That is, if $M$ is a rank $k$ matroid with a coloop $e$, then $M$ is threshold if and only if $M/e$ is threshold.
\end{lemma}
\begin{proof}
First suppose that $M/e$ is threshold with weight function $w:E\setminus e \to \R$. Let $\alpha = \max(w)+1$ and $\tilde w: E \to \R$ be defined by
\[\tilde w(i) = \begin{cases}
                    (k-1)\alpha & \text{if }i=e\\
                    w(i)-\alpha & \text{if }i\neq e
\end{cases}.\]
We claim that $\tilde w$ is a weight function for $M$. Since $w(i)-\alpha< 0$ for any $i\neq e$, a size $k$ subset $A\subseteq E$ which does not contain $e$ will have $\tilde w(A)< 0$. Now suppose that $e\in A$. Then
\[\tilde w (A) = (k-1)\alpha + w(A\setminus e) - (k-1)\alpha = w(A\setminus e). \]
This means that $\tilde w(A)$ will have positive weight if and only if $A$ contains $e$ and $A\setminus e$ is a basis of $M/e$. In other words, $\tilde w(A)$ will have positive weight if and only if $A$ is a basis of $M$.

Now suppose that $M$ is threshold with weight function $w:E\to \R$. Let $\tilde w:E\setminus e\to \R$ be defined by $\tilde w(i) = w(i) + \frac{w(e)}{k-1}$. Then for any size $k-1$ subset $A\subseteq E\setminus e$,
\[\tilde w(A) = w(A) + w(e) = w(A\cup e). \]
Since $w$ is a threshold function for $M$, this means that $\tilde w(A)$ is positive if and only if $A\cup e$ is a basis of $M$.
\end{proof}

\section{Shifted vs Threshold}
\label{sec:counter}
Although seemingly distinct classes of matroids, it turns out that all threshold matroids are shifted. The weight function of a threshold matroid $M$ induces an ordering on the ground set for which $M$ is a shifted matroid. The more general version of this statement for $k$-hypergraphs has been observed many times; see \cite[Theorem $5.3.1$]{Muroga}, \cite[Lemma $2.1$]{DO} and \cite[Theorem $9.8$]{Boolean}. Because its converse statement is central to our study, we include a proof of the fact that all threshold matroids are shifted.

\begin{lemma}
  \label{lem:shift->thresh}
All threshold matroids are shifted.
\end{lemma}
\begin{proof}
  Suppose $M$ is a threshold matroid with weight function $w:E\to \R$. Let $<$ be a total ordering of $E$ such that $x<y$ whenever $w(x)> w(y)$. Now consider the order preserving bijection $\psi:E\to [n]$ defined by $<$. We show that the image of $M$ under $\psi$ forms an order ideal in the component-wise partial order. Suppose $B\subseteq [n]$ is the image of a basis of $M$. We need to check that if $b\in B$, $d\not\in B$ and $d<b$, then $(B\setminus b) \cup d$ is also the image of a basis of $M$. But because $d<b$, $w(\psi^{-1}d)>w(\psi^{-1}b)$ so \[w(\psi^{-1}((B\setminus b) \cup d)) = w(\psi^{-1}B)+w(\psi^{-1}d)-w(\psi^{-1}b)>0\]
  and our claim follows.
\end{proof}

Since all threshold matroids are shifted, it is natural to wonder whether the converse might be true. This question was asked for matroids, using slightly different terminology, by Deza and Onn. For a translation between the language used in \cite{DO} and in this article, see Table \ref{tab:table}.
\begin{question}[{\cite[Question $7.2$]{DO}}]
\label{q:thresh}
Are all shifted matroids threshold?
\end{question}

This question was posed in the broader context of shifted and threshold $k$-hypergraphs. It is well known that not all shifted $k$-hypergraphs are threshold. See, for example,  \cite[Theorem $3.1$]{KR}, \cite[Example $2.2$]{RRST}, \cite[Theorem $7.3.3$]{Muroga} and \cite[Example $2.15$]{BS}. Detecting which $k$-hypergraphs are threshold is the focus of much research; see \cite[Chapter $9$]{Boolean} and its references therein. Deza and Onn ask for which families of $k$-hypergraphs are shiftedness and thresholdness the same? They postulate that this might be true for matroids. As evidence, they prove that shifted paving matroids, shifted binary matroids and shifted rank $3$ matroids are threshold \cite[Theorems $5.1$, $5.4$ and $5.5$]{DO}. They also note that none of the previously known examples of shifted but non-threshold $k$-hypergraphs are a matroid \cite[Section $6$]{DO}. 

However, in the example below, we give an example of a shifted matroid that is not threshold. This is a counterexample to Question \ref{q:thresh} and shows that thresholdness is a stronger condition than shiftedness. In fact, as we will show in Theorem \ref{thrm:almost}, thresholdness is a \emph{much} stronger condition than shiftedness.

\begin{exthrm}[Counterexample to Question \ref{q:thresh}]
\label{ex:counter}
  There exists a shifted matroid that is not threshold.
\end{exthrm}
\begin{proof}
Let $M=\langle 2468\rangle$ be the shifted matroid with defining basis $2468\subseteq [8]$. Suppose that $M$ is threshold with weight function $w:[8]\to \R$. Note that 
\[2468,1357 \in M \text{ and } 1278,3456 \not\in M \]
so
\[w(2468),w(1357) > 0 \text{ and } w(1278), w(3456) \leq 0. \]
But both $2468\sqcup 1357$ and $1278 \sqcup 3456$ partition $[8]$, so we must have 
\[\hspace{4.5cm}0< w(2468) + w(1357) = w(1278)+w(3456) \leq 0. \hspace{4cm}\]
\end{proof}

Counterexample \ref{ex:counter} inspired the characterizations of threshold matroids given in Theorem \ref{thrm:main} and Theorem \ref{thrm:2-trade}. After proving Theorem \ref{thrm:main}, we will be able to see that $M=\langle 2468 \rangle$ is not threshold because $2468$ has four blocks. After proving Theorem \ref{thrm:2-trade}, we will be able to see that $M=\langle 2468 \rangle$ is not threshold because $2468$ and $1357$ are bases, $1278$ and $3456$ are non-bases, $2468\cup 1357= [8]=1278 \cup 3456$ and $2468\cap 1357 = \emptyset = 1278\cap 3456$.

\section{Uniform Asummability}
\label{sec:trade}

We now introduce our main tool for analyzing which shifted matroids are threshold. We call this tool \emph{uniform asummability} as it is a uniform analogue of asummability from the simple games literature \cite[Section $2.6$]{TZ} and the Boolean functions literature \cite[Chapter $9$]{Boolean}. To introduce uniform asummability, we first equip the set $W(n)$ with a binary operation to turn it into a graded commutative monoid. 

\begin{defn}
  \label{defn:sorted}
  Let $A\in W(k,n)$ and $B\in W(k',n)$. We define the \emph{sorted concatenation} of $A$ and $B$, denoted $A+B$, as the length $k+k'$ weakly increasing word constructed by first concatenating $A$ and $B$ and then rearranging their concatenation into weakly increasing order. Formally, the weakly increasing function $A+B:[k+k']\to [n]$ is defined by the property that $|(A+B)^{-1}(i)|=|A^{-1}(i)|+|B^{-1}(i)|$ for all $i\in [n]$.
\end{defn}

We can also think of $A+B$ as the word obtained from rearranging the multiset union of $A$ and $B$ into weakly increasing order. For example, $234+134= 123344$.

\begin{defn}
\label{def:trade_robust}
Let $\ell$ be a positive integer. A matroid $M$ is \emph{$\ell$-uniform asummable} if for all $B_1,\ldots, B_\ell \in M$ and $D_1,\ldots, D_\ell \in \binom{[n]}{k}\setminus M$ we have
\[B_1+\ldots+B_\ell \neq D_1+\ldots +D_\ell. \]
Note that we do not require the set of bases $\{B_1,\ldots, B_\ell\}$ or the set of non-bases $\{D_1,\ldots, D_\ell\}$ to be pairwise distinct. If $M$ is $\ell$-uniform asummable for all $\ell$, then we say $M$ is \emph{uniform asummable}.
\end{defn}

Note that if a matroid $M$ is $\ell$-uniform asummable then $M$ is $k$-uniform asummable for all $k<\ell$. Uniform asummability is important to the present work because of the following special case of a theorem of Klivans and Reiner.

\begin{thrm}[Theorem 3.1 of \cite{KR}]
  \label{thrm:KR}
A matroid $M$ is threshold if and only if it is uniform asummable.
\end{thrm}

\begin{remark}
  In Klivans and Reiner's work, $\ell$-uniform asummability is referred to as the $CC_\ell$ property. They attribute Theorem \ref{thrm:KR} to Taylor and Zwicker \cite[Theorem $2.4.2$]{TZ}. However, Taylor and Zwicker's result is a different, but related, result. Taylor and Zwicker's result, proved in the language of simple games, specializes to the following statement for matroids: A matroid $M$ is \emph{non-uniform threshold} if and only if it is \emph{non-uniform asummable}. See Remark \ref{rem:GK} for a discussion of non-uniform thresholdness. The definition of non-uniform asummability is identical to uniform asummability except that we no longer require the $B$'s and $D$'s to have size $k$. Non-uniform asummability is a strictly stronger condition than uniform asummability. For example, the shifted matroid $M= \langle 246 \rangle$ is uniform asummable but is not $2$-non-uniform asummable because $135 + 246 = 56 + 1234$.
\end{remark}

If one wants to exhibit a matroid that is not threshold, it suffices to present $\ell$ many bases $B_1,\ldots, B_\ell$ and $\ell$ many non-bases $D_1,\ldots, D_\ell$ so that 
\[B_1+\ldots+B_\ell = D_1+\ldots+D_\ell. \]

\begin{example}
We show that the shifted matroid $M=\langle 2468 \rangle$ is not $2$-uniform asummable and hence, as seen previously in Example \ref{ex:counter}, not threshold. Let
\[B_1= 2468, B_2= 1357, D_1=1278 \text{ and } D_2=3456. \]
Then $B_1,B_2 \in M$ and $D_1,D_2 \in \binom{[n]}{k}\setminus M$ but $B_1+B_2 = [8]=D_1+D_2$. Thus, $M$ is \emph{not} $2$-uniform asummable and is not threshold. 
\end{example}

\begin{lemma}
\label{lem:ele}
A matroid $M$ is $2$-uniform asummable if and only if for all bases $B_1,B_2\in M$ and non-bases $D_1,D_2\in \binom{[n]}{k}\setminus M$, we have that
\[B_1\cup B_2 \neq D_1\cup D_2 \text{ or } B_1\cap B_2 \neq D_1 \cap D_2. \]
\end{lemma}
\begin{proof}
The sorted concatenations of two pairs $B_1,B_2$ and $D_1,D_2$ are equal if and only if $B_1 \cup B_2 = D_1\cup D_2$ and $B_1\cap B_2 = D_1\cap D_2$. Thus, our second condition is simply a rephrasing of the definition of $2$-uniform asummability. 
\end{proof}

It is asked in the work of Deza and Onn \cite[Question $1.1$]{DO} whether the property outlined in Lemma \ref{lem:ele} characterizes threshold matroids. Our Theorem \ref{thrm:2-trade} gives a positive answer to this question. We will prove Theorem \ref{thrm:2-trade} in the course of proving Theorem \ref{thrm:main}. 

\begin{remark}
  The relationship between $2$-non-uniform asummability and non-uniform thresholdness has been the subject of much study in the theories of simple games and Boolean functions; see \cite[Section $2.7$]{TZ} and \cite[Section $9.3$]{Boolean}. In general, the class of $2$-non-uniform asummable hypergraphs strictly contains the class of non-uniform threshold hypergraphs. However, as a consequence of Giles and Kannan's result \cite{GK}, $2$-non-uniform asummable matroids are the same as non-uniform threshold matroids.

  Theorem \ref{thrm:2-trade} shows that a similar story is true in the uniform setting. Although $2$-uniform asummability does not imply thresholdness for hypergraphs \cite[Section $6$]{DO}, the class of threshold matroids is equal to the class of $2$-uniform asummable matroids.
\end{remark}

The goal of the rest of this section is to further refine uniform asummability for shifted matroids. We first show that we lose nothing by restricting to shifted matroids. That is, every $\ell$-uniform asummable matroid is shifted.

\begin{lemma}
  \label{lem:assum->shifted}
Suppose $M$ is an $\ell$-uniform asummable matroid for some $\ell>1$. Then $M$ is a shifted matroid.
\end{lemma}
\begin{proof}
  If $M$ is $\ell$-uniform asummable for some $\ell>1$, then $M$ is $2$-uniform asummable. We now proceed by contrapositive to show that $2$-uniform asummable matroids are shifted. Suppose $M$ is not shifted. Then there exist bases $B_1,B_2$ and elements $b_1\in B_1\setminus B_2$, $b_2\in B_2\setminus B_1$ such that $(B_1\setminus b_1)\cup b_2\not\in M$ and $(B_2\setminus b_2)\cup b_1 \not\in M$. But
  \[B_1+B_2 = (B_1\setminus b_1)\cup b_2 +(B_2\setminus b_2)\cup b_1 \]
  so $M$ is not $2$-uniform asummable. 
\end{proof}

We now analyze the relationship between sorted concatenation and the component-wise partial order.

\begin{lemma}\label{lem:prec_conc}
Suppose that $A,A'\in W(k_A,n)$, $B,B'\in W(k_B,n)$, $A\preceq A'$ and $B\preceq B'$. Then $A+B\preceq A'+B'$ in $W(k_A+k_B,n)$. 
\end{lemma}
\begin{proof}
  Let $\iota_A:[k_A]\hookrightarrow [k_A+k_B]$ and $\iota_B:[k_B]\hookrightarrow [k_A+k_B]$ be injections with the following properties: the images of $\iota_A$ and $\iota_B$ partition $[k_A+k_B]$, $A=(A+B)\circ\iota_A$ and $B=(A+B)\circ\iota_B$. Define $\iota_{A'}$ and $\iota_{B'}$ similarly for $A'$ and $B'$. Let $\pi:[k_A+k_B]\to [k_A+k_B]$ be the bijection which sends $\iota_{A}(i)\mapsto \iota_{A'}(i)$ and $\iota_{B}(j)\mapsto \iota_{B'}(j)$ for all $i\in[k_A]$ and $j\in[k_B]$. Then
  \[(A'+B')\circ\pi(\alpha) = \begin{cases}
                       (A'+B')\circ \iota_{A'}(i) = {A'}(i), & \text{if $\alpha=\iota_A(i)$ for $i\in[k_A]$}\\
                       (A'+B')\circ \iota_{B'}(j) = {B'}(j), & \text{if $\alpha=\iota_B(j)$ for $j\in[k_B]$.}
                     \end{cases}\]
But ${A'}(i)\geq A(i)=(A+B)\circ\iota_A(i)$ and ${B'}(j)\geq B(j)=(A+B)\circ\iota_B(j)$ for all $i\in[k_A]$ and $j\in[k_B]$. Thus $(A'+B')\circ \pi(\alpha)\geq (A+B)(\alpha)$ for all $\alpha\in[k_A+k_B]$. By Lemma \ref{lem:weaker_prec}, this implies that $A+B\preceq A'+B'$.
\end{proof}

The converse of Lemma \ref{lem:prec_conc} is not always true. For example, neither $1278$ nor $3456$ is smaller than $2468$ but
\[1278 +  3456 \leq 2468 + 2468. \]
As we will see in Lemma \ref{lem:robust_leq}, uniform asummability of a shifted matroid $M=\langle T \rangle$ measures to what extent the converse of Lemma \ref{lem:prec_conc} is true for the non-bases of $M$.

If $M$ is a shifted matroid, we can use the component-wise partial order to determine uniform asummability. This is the formulation which will be most useful for us. We will use this to determine which shifted matroids are \emph{not} threshold; see Theorem \ref{thrm:main}.

\begin{lemma}\label{lem:robust_leq}
  A rank $k$, shifted matroid $M=\langle T\rangle$ is $\ell$-uniform asummable if and only if for any $D_1,D_2,\ldots, D_\ell\in \binom{[n]}{k}\setminus M$,
  \[D_1+ \ldots + D_\ell \npreceq T+ \ldots + T \]
  where the right-hand side is $\ell$ copies of $T$.
\end{lemma}
\begin{proof}
  We prove both directions by contrapositive. First, suppose that $M$ is not $\ell$-uniform asummable. Let $B_1,\ldots,B_\ell\in M$ and $D_1,\ldots,D_\ell\not\in M$ be witnesses of this failure so that, 
  \[B_1+\ldots+B_\ell = D_1+\ldots+D_\ell. \]
  Each $B_i\preceq T$, so Lemma \ref{lem:prec_conc} ensures that 
  \[D_1+ \ldots + D_\ell = B_1+ \ldots + B_\ell \preceq T+ \ldots + T. \]
  
  \noindent Now for the other direction, assume that $D_1,D_2,\ldots, D_\ell \in \binom{[n]}{k}\setminus M$ but 
    \[ D_1+ \ldots + D_\ell  \preceq T+ \ldots + T. \]
   We will construct bases $B_1,\ldots, B_\ell$ such that $B_1+\cdots + B_\ell = D_1+ \cdots + D_\ell$, thus showing that $M$ is not $\ell$-uniform asummable. Let $T= t_1 t_2 \cdots t_k$ and 
   \[T+\ldots + T = t_{1,1}t_{1,2}\cdots t_{1,\ell}t_{2,1}\cdots t_{k,\ell}\]
    where $t_{i,j}=t_{i,j'}=t_i$ for any $i$ and any $j,j'$. Also let $D_1+\ldots + D_\ell = b_{1}\cdots b_{k\cdot\ell}$. By our assumption, $b_{\ell (i-1) + j}\leq t_{i,j}$ for all $1\leq i\leq k$ and $1\leq j \leq \ell$. Now let
    \[ B_j = \{ b_{\ell (i-1) + j}: 1\leq i \leq k\}\]
    for all $1\leq j \leq \ell$. Because $t_i=t_{i,j}\geq  b_{\ell (i-1) + j}$, we have that $B_j\preceq T$. Further, by definition of the $B_j$s,
    \[B_1+\ldots + B_\ell = D_1+\ldots +D_\ell.\]
    Now the only thing left to check is that each $B_j$ is an honest set, i.e., each $B_{j}$ does not contain any repeated elements. Suppose that some $B_j$ has a repeated element. This means that there is some index $1\leq i < k$ so that $b_{\ell (i-1) +j} = b_{\ell i +j}$. But because the $b_i$s are arranged in weakly increasing order, this implies
    \[b_{\ell(i-1)+j} = b_{\ell(i-1)+j+1} = \ldots = b_{\ell(i-1)+j+\ell}. \]
    Hence, $D_1+\ldots+ D_\ell$ contains at least $\ell+1$ repeated elements. Since each $D_i$ is a set, this would violate the pigeonhole principle and lead to a contradiction.
\end{proof}

\section{Proof of the main theorem}
\label{sec:main}

Here we prove our characterization of threshold matroids. This characterization determines whether a shifted matroid is threshold using only the data of its defining basis. We will explore the consequences of this characterization in Section \ref{sec:consequences}.
\begin{repintrotheorem}{thrm:main}
Suppose $M$ is a shifted matroid and let $T$ be the defining basis of the shifted matroid obtained from contracting all coloops of $M$.
Then:
\begin{enumerate}[i)]
\item \label{thrm:main_i}If $T$ has four or more blocks, $M$ is not threshold.
\item \label{thrm:main_ii}If $T$ has two or fewer blocks, $M$ is threshold.
\item If $T$ has three blocks:
\begin{enumerate}[a)]
\item \label{thrm:main_iiia}If the second block or gap of $T$ has size one, $M$ is threshold.
\item \label{thrm:main_iiib}Otherwise, $M$ is not threshold. 
\end{enumerate}
\end{enumerate}
\end{repintrotheorem}
By Lemma \ref{lem:thresh_contract}, $M$ is threshold if and only if the matroid with defining basis $T$ is threshold. Thus, for the rest of this section, we will assume that $M=\langle T \rangle$ and $M$ contains no coloops. This is equivalent to the condition that $1\not \in T$.

We prove this theorem in two parts. For the cases where $M$ is threshold, we exhibit a threshold function. When $M$ is not threshold, we use Lemma \ref{lem:robust_leq} to show that $M$ is not $2$-uniform asummable (Definition \ref{def:trade_robust}) and hence not threshold. Thus, as part of this proof, we obtain Theorem \ref{thrm:2-trade}. To aid in understanding, we include explicit examples throughout our proofs.
\begin{reptheorem}{thrm:main_i}
\label{prop:4blocks}
If $M=\langle T\rangle$ is a coloopless shifted matroid of rank $k$ on ground set $[n]$ and $T$ has four or more blocks, then $M$ is not $2$-uniform asummable (Definition \ref{def:trade_robust}) and hence not threshold. 
\end{reptheorem}
\begin{proof}
  If $T$ has more than four blocks, $M$ contains a basis $T'$ that has exactly four blocks. Such a basis can be obtained by shifting all the blocks of $T$ that come after the fourth block leftward so that the fourth block of $T'$ is one long block. For example, if \[T= 2\,3\,4\quad7\,8\,9\quad 11\,12 \quad 14\quad 16\quad 18,\] then \[T'=2\,3\,4\quad 7\,8\,9\quad 11\,12\quad 14\,15\,16.\]

  We will now construct two sets $D_1$ and $D_2$ such that $D_1+ D_2 \preceq T+T$. We will later show that $D_1$ and $D_2$ are non-bases of $M$. Suppose that the $i$th block of $T'$ has size $\lambda_i$ and begins with the element $x_i$. Define $D_1$ and $D_2$ by declaring
  \[D_1 = T'\cup \{x_1-1,x_4-1\} \setminus \{x_2,x_3\}\quad\text{and}\quad D_2= T'\cup \{x_2-1,x_3-1\} \setminus \{x_1,x_4\}.\]
  As each $x_i$ is the start of a block of $T'$, $D_1$ and $D_2$ are indeed sets. The colooplessness of $M$ guarantees that $x_1-1\geq 1$, so both $D_1$ and $D_2$ are subsets of $[n]$. By their construction, we have that $D_1+D_2 \preceq T'+T'\preceq T+T$. In our example,
\[
T'+T' = \left\{ \begin{aligned} {\color{blue}2\,3\,4\quad 7\,8\,9\quad 11\,12\quad  14\,15\,16}\\
\underline{\color{red}2}\,\underline{\color{red}3}\,\underline{\color{red}4}\quad \underline{\color{red}7}\,\underline{\color{red}8}\,\underline{\color{red}9}\quad \underline{\color{red}11}\,\underline{\color{red}12}\quad \underline{\color{red}14}\,\underline{\color{red}15}\,\underline{\color{red}16}
\end{aligned} \right\}
\]
and
\begin{gather*}
D_1 = \underline{\color{red}1}\,{\color{blue}2\,3\, 4\quad 8\,9\quad 12}\quad \underline{\color{red}13}\, {\color{blue}14\,15\,16}\\
D_2 = \underline{\color{red}3}\,\underline{\color{red}4}\quad {\color{blue}6}\, \underline{\color{red}7}\,\underline{\color{red}8}\,\underline{\color{red}9}\quad{\color{blue} 10}\,\underline{\color{red}11}\, \underline{\color{red}12}\quad \underline{\color{red}15}\,\underline{\color{red}16}.
\end{gather*}
The underlining and color-coding in our example indicates a weakly increasing bijection between the letters of $D_1+D_2$ and the letters of $T'+T'$. Such a bijection is witness to the fact that $D_1+D_2 \preceq T'+T'$.

We finish by showing that $D_1$ and $D_2$ are non-bases of $M$. The $\lambda_1$st element of $D_2$ is equal to $x_2-1$, while the $\lambda_1$st element of $T'$ is the last element of the first block of $T'$ and hence is bounded above by $x_2-2$. Therefore, $D_2$ is larger than $T'$ at index $\lambda_1$. By our construction of $T'$, $T'$ is equal to $T$ at index $\lambda_1$. Hence, $D_2$ is larger than $T$ at index $\lambda_1$ and $D_2\not\in M$. 

Similarly, the $(\lambda_1+\lambda_2+\lambda_3)$rd element of $D_1$ is equal to $x_4-1$, while the $(\lambda_1+\lambda_2+\lambda_3)$rd element of $T'$ is the last element of the third block of $T'$ and is bounded above by $x_4-2$. Thus, we have that $D_1$ is larger than both $T'$ and $T$ at index $(\lambda_1+\lambda_2+\lambda_3)$ and therefore $D_1\not\in M $.

In our example, these observations correspond to the fact that $D_1$ and $D_2$ are larger than $T$ at the red underlined indices below:
\begin{align*}
D_1 = 1\,2\,3\quad 4\, 8\,9\quad 12\,&{\color{red} \underline{13}}\quad 14\,15\,16\\
T\,\,= 2\,3\,4\quad7\,8\,9\quad 11\,&{\color{red}\underline{12}} \quad 14\quad 16\quad 18
\end{align*}
and 
\begin{align*}
D_2 = 3\,4\, &{\color{red}\underline{6}}\quad 7\,8\,9\quad 10\,11\quad 12\, 15\,16\\
T\,\,= 2\,3\,&{\color{red}\underline{4}}\quad7\,8\,9\quad 11\,12 \quad 14\quad 16\quad 18.
\end{align*}
\end{proof}

We give a similar proof for the other case of non-thresholdness. We will follow the same strategy, but the details will be altered slightly.

\begin{repdoubletheorem}{thrm:main_iiib}
\label{prop:3blocks}
Let $M=\langle T\rangle$ be a coloopless shifted matroid of rank $k$. If $T$ has $3$ blocks and the second block and second gap of $T$ both have size at least two, then $M$ is not $2$-uniform asummable (Definition \ref{def:trade_robust}) and hence not threshold. 
\end{repdoubletheorem}
\begin{proof}
  Suppose that the $i$th block of $T$ has length $\lambda_i$ and begins with the element $x_i$. We again construct two sets $D_1$ and $D_2$ such that $D_1+D_2\preceq T+T$ and show that $D_1$ and $D_2$ are non-bases of $M$. Let $D_1$ and $D_2$ be defined by
  \[D_1 = T\cup \{x_1-1,x_3-1\}\setminus \{x_2,x_2+1\} \quad\text{and}\quad D_2 = T\cup \{x_2-1,x_2-2\}\setminus \{x_1,x_3\}. \]
As the second block and second gap of $T$ both have size at least two, both $D_1$ and $D_2$ are sets of size $k$. The colooplessness of $M$ guarantees that $x_1-1\geq 1$, so both $D_1$ and $D_2$ are subsets of $[n]$. By their construction, $D_1+D_2 \preceq T+T$.

  For a running example, suppose that 
\[ T = 2\,3\,4\quad 7\,8\,9\quad 11\,12.\]
Then
\[
T+T = \left\{ \begin{aligned} {\color{blue}2\,3\,4\quad 7\,8\,9\quad 11\,12}\\
\underline{\color{red}2}\,\underline{\color{red}3}\,\underline{\color{red}4}\quad \underline{\color{red}7}\,\underline{\color{red}8}\,\underline{\color{red}9}\quad \underline{\color{red}11}\,\underline{\color{red}12}\end{aligned}\right\} \quad \text{and}\,\quad
{\begin{aligned}
&D_1 = \underline{\color{red}1}\,{\color{blue} 2\, 3\, 4\quad 9}\quad \underline{\color{red}10}\,  {\color{blue}11\,12}\\
&D_2 = \underline{\color{red}3}\, \underline{\color{red}4}\quad {\color{blue}5\, 6}\, \underline{\color{red}7}\,  \underline{\color{red}8}\, \underline{\color{red}9}\quad \underline{\color{red}12}.
\end{aligned}}
\]

As before, the underlining and color-coding in our example indicates a weakly increasing bijection between the letters of $D_1+D_2$ and the letters of $T+T$. Such a bijection is witness to the fact that $D_1+D_2 \preceq T+T$.

Finally, we prove that $D_1$ and $D_2$ are non-bases of $M$. The $\lambda_1$st element of $D_2$ is equal to $x_2-2$ while the $\lambda_1$st element of $T$ is the last element of the first block of $T$. This element is bounded above by $x_2-3$ since we assumed that the second gap of $T$ is has size at least two. Therefore, $D_2$ is larger than $T$ at index $\lambda_1$ and $D_2\not\in M$.

Similarly, the $(\lambda_1+\lambda_2)$nd element of $D_1$ is equal to $x_3-1$. But the $(\lambda_1+\lambda_2)$nd element of $T$ is the last element of the second block of $T$ and is bounded above by $x_3-2$. Therefore, we have that $D_1$ is larger than $T$ at index $\lambda_1+\lambda_2$ and $D_1\not\in M$.

In our example, these observations correspond to the fact that $D_1$ and $D_2$ are larger than $T$ at the red underlined indices below:
\begin{align*}
D_1 = 1\, 2\, 3\quad  4\, 9\, &{\color{red}\underline{10}}\quad  11\,12\\
T=2\,3\,4\quad 7\,8\,&{\color{red} \underline 9}\quad 11\,12
\end{align*}
and 
\begin{align*}
D_2 = 3\, 4\, &{\color{red} \underline 5}\quad 6\, 7\,  8\quad 9\, 12\\
T=2\,3\,&{\color{red} \underline 4}\quad 7\,8\,9\quad 11\,12.
\end{align*}
\end{proof}

We can now proceed to prove the cases where $M$ is threshold. We do so by explicitly constructing weight functions. We start with the simplest case.
\begin{reptheorem}{thrm:main_ii}
If $M=\langle T\rangle$ is a shifted matroid of rank $k$ and $T$ has two or fewer blocks, then $M$ is threshold.
\end{reptheorem}
\begin{proof}
Suppose $T$ only has one block and that this block ends at some $t_0\in [n]$. Then every element $s>t_0$ is a loop of $M$. After deleting all such $s>t_0$ from $M$, we are left with the rank $k$ uniform matroid on the ground set $[t_0]$. Now, note that every uniform matroid is a threshold matroid. To see this, observe that any strictly positive function is a weight function for a uniform matroid. As uniform matroids are threshold and adding loops preserves thresholdness \cite[Lemma $5.2$]{DO}, we can conclude that $M=\langle T \rangle$ is threshold.

Now, suppose that $T$ has two blocks, the first block ending at $t_1$ with size $\lambda_1$ and the second block ending at $t_2$ with size $\lambda_2$. Then a size $k$ subset $A\subseteq [n]$ is a basis of $M$ if and only if the following conditions hold:
\begin{itemize}
\item $A$ has no elements larger than $t_2$.
\item $A$ has at least $\lambda_1$ elements smaller than $t_1$.
\end{itemize}
We can model this using the weight function $w:[n]\to \R$ defined by
\[w(j) = \begin{cases}
 \lambda_2 & \text{if $j\leq t_1$}\\
 -\lambda_1 & \text{if $t_1< j \leq t_2$}\\
 -\infty & \text{if $j> t_2$}
\end{cases}. \]
By our previous description, one can verify that $M$ is indeed the matroid associated to $w$ and hence $M$ is threshold.
\end{proof}

We now prove the last remaining cases of thresholdness. We will prove one part of the case explicitly and then use the fact that thresholdness is closed under duality to prove the other.

\begin{lemma}
\label{prop:3block_w}
Let $M=\langle T\rangle$ be a coloopless shifted matroid of rank $k$. If $T$ has three blocks and the second block of $T$ has size one, then $M$ is threshold.
\end{lemma}
\begin{proof}
Suppose the first block of $T$ ends at $t_1$ and has size $\lambda_1$, the second block of $T$ consists solely of the element $t_2$ and the third block of $T$ ends at $t_3$ and has size $\lambda_3$. Let $w:[n]\to \R$ be defined by 
\[ w(j) = \begin{cases}
2\lambda_3^3+6\lambda_3^2+4\lambda_3+1 & \text{ if } j \leq t_1\\
(1-2\lambda_1)\lambda_3^2+(1-3\lambda_1)\lambda_3 & \text{ if } t_1< j \leq t_2\\
-2\lambda_1\lambda_3^2-(4\lambda_1+1)\lambda_3-(\lambda_1+1) & \text{ if } t_2 < j \leq t_3\\
-\infty & \text { if } t_3< j
\end{cases}.\]
We will prove that $w$ is a weight function for $M$ and hence that $M$ is threshold. To do so, we first show that $w$ is a weakly monotonically decreasing function. This reduces the task of checking that $w$ is non-negative on all bases of $M$, to just checking that $w$ is non-negative on $T$. We then use our characterization of circuits, provided by Lemma \ref{lem:circuits}, to determine the structure of the non-bases of $M$. We use this structure and the fact that $w$ is decreasing to verify that $w$ is negative on all non-bases of $M$.  

To see that $w$ is weakly monotonically decreasing, observe that
\[
\left(2\lambda_3^3+6\lambda_3^2+4\lambda_3+1\right) - \left((1-2\lambda_1)\lambda_3^2+(1-3\lambda_1)\lambda_3\right) =
2\lambda_3^3 + (2\lambda_1+5)\lambda_3^2 + (3\lambda_1+3)\lambda_3+1>0\]
and
\[\left((1-2\lambda_1)\lambda_3^2+(1-3\lambda_1)\lambda_3\right) -\left(-2\lambda_1\lambda_3^2-(4\lambda_1+1)\lambda_3-(\lambda_1+1)\right)=
\lambda_3^2+\lambda_1\lambda_3 + \lambda_1+1>0.
\]
We can now compute $w(T)$ to check that the weight of $T$ is positive: 
\begin{align*}
w(T)&= \lambda_1\left(2\lambda_3^3+6\lambda_3^2+4\lambda_3+1\right) + 1\left((1-2\lambda_1)\lambda_3^2+(1-3\lambda_1)\lambda_3\right) + \lambda_3\left(-2\lambda_1\lambda_3^2-(4\lambda_1+1)\lambda_3-(\lambda_1+1)\right)\\
&= (2\lambda_1-2\lambda_1)\lambda_3^3 + (6\lambda_1+1-2\lambda_1-4\lambda_1-1 )\lambda_3^2 + (4\lambda_1+1-3\lambda_1-\lambda_1-1)\lambda_3+ \lambda_1\\
    &= \lambda_1\\
  &> 0.
\end{align*}

\noindent By Lemma \ref{lem:circuits}, the non-loop circuits of $M$ with size less than $k$ are of the form 
\[C = c_0  c_{1}  c_{2} \ldots  c_{\lambda_3}\]
and 
\[D = d_1  d_{2}  d_{3} \ldots  d_{\lambda_3+2}\]
where: 
\begin{itemize}
\item $c_0>t_2$ and $c_{1} c_{2} \ldots c_{\lambda_3}$ is smaller than the third block of $T$.
\item $d_1>t_1$,  $d_{2}$ is smaller than $t_2$ and  $d_{3} \ldots d_{\lambda_3+2}$ is smaller than the third block of $T$.
\end{itemize}
Since any non-basis must contain a circuit, a non-basis must contain a loop or a circuit of the form $C$ or $D$. If a non-basis contains a loop, then it contains some element larger than $t_3$ and hence has a negative weight. We now consider non-bases that don't contain a loop. First, suppose that $A$ is a size $k$ set and contains a circuit of the form $C$. Then $A$ contains at least $\lambda_3+1$ many elements greater than $t_2$. By the monotonicity of $w$, this implies that $w(A) \leq (k-\lambda_3-1)w(1) + (\lambda_3+1)w(t_2+1)$. Therefore, 
\begin{align*}
w(A) &\leq (k-\lambda_3-1)w(1) + (\lambda_3+1)w(t_2+1)\\
     &= \lambda_1\left(2\lambda_3^3+6\lambda_3^2+4\lambda_3+1\right) + (\lambda_3+1)\left(-2\lambda_1\lambda_3^2-(4\lambda_1+1)\lambda_3-(\lambda_1+1)\right)\\
     &= -\lambda_3^2-(\lambda_1+2)\lambda_3-1\\
     &< 0.
\end{align*}
Similarly, suppose $A'$ is a size $k$ set and contains a circuit of the form $D$. Then $A'$ contains at least $\lambda_3+2$ elements larger than $t_1$. By the monotonicity of $w$, this implies that $w(A') \leq (k-\lambda_3-2)w(1) + (\lambda_3+2)w(t_1+1)$. Therefore, 
\begin{align*}
w(A') &\leq (k-\lambda_3-2)w(1) + (\lambda_3+2)w(t_1+1)\\
      &= (\lambda_1-1)\left(2\lambda_3^3+6\lambda_3^2+4\lambda_3+1\right) + (\lambda_3+2)\left((1-2\lambda_1)\lambda_3^2+(1-3\lambda_1)\lambda_3\right)\\
      &= -\lambda_3^3-(\lambda_1+3)\lambda_3^2-(\lambda_1+2)\lambda_3-(\lambda_3-1)\lambda_1-1\\
      &< 0.
\end{align*}

We have now shown that $w$ is non-negative on all bases of $M$ and negative on all non-bases of $M$ and thus that $w$ is a threshold function for $M$. 
\end{proof}

Since thresholdness is preserved under duality, we obtain the last case of the main theorem. 

\begin{corollary}
Let $M=\langle T\rangle$ be a coloopless shifted matroid of rank $k$. If $T$ has three blocks and the second gap of $T$ has size one, then $M$ is threshold.
\end{corollary}
\begin{proof}
The dual matroid $M^\perp$ of $M$ is the shifted matroid determined by the set $[n]\setminus T$ with the total ordering of $[n]$ reversed so $n<n-1<\ldots < 1$. Under this reordering, $M^\perp$ will have coloops if and only if $M$ has loops. First, suppose that $M$ has no loops, so that $M^\perp$ has no coloops. Then $M^\perp$ will be a shifted matroid with three blocks whose second block has size one. By Proposition \ref{prop:3block_w}, $M^\perp$ will be threshold and, by Lemma \ref{lem:thresh_dual}, $M$ will be threshold. 

Now suppose $M$ has loops so that $M^\perp$ has coloops. By Lemma \ref{lem:thresh_contract}, we can contract the coloops from $M^\perp$ without altering whether or not it is threshold. By contracting all coloops of $M^\perp$, we obtain a shifted matroid with three blocks whose second block has size one. We can then proceed as in the previous case to show that $M$ is threshold.
\end{proof}

Combining the previous lemma and corollary, we obtain the last part of our main theorem.
\begin{repdoubletheorem}{thrm:main_iiia}
Let $M=\langle T\rangle$ be a coloopless shifted matroid of rank $k$. If $T$ has three blocks and the second block or gap of $T$ has size one, then $M$ is threshold.
\end{repdoubletheorem}

Finally, we note that our proof of Theorem \ref{thrm:main} actually proves Theorem \ref{thrm:2-trade} as well.

\begin{repintrotheorem}{thrm:2-trade}
A matroid is threshold if and only if it is $2$-uniform asummable. Equivalently, a rank $k$ matroid $M$ is threshold if and only if for all bases $B_1,B_2$ and size $k$ non-bases $D_1,D_2$, we have 
\[B_1 \cup B_2 \neq D_1 \cup D_2 \text{ or } B_1\cap B_2 \neq D_1 \cap D_2. \]
\end{repintrotheorem}
\begin{proof}
  By Lemmas \ref{lem:shift->thresh} and \ref{lem:assum->shifted}, any threshold or $2$-uniform asummable matroid is shifted. Furthermore, it follows from the definition that $2$-uniform asummability is invariant under contracting coloops. Thus, we can now restrict to the class of shifted matroids without coloops. But now the result follows from Theorems \ref{thrm:main}, \hyperref[thrm:main_i]{2.\ref*{thrm:main_i}} and \hyperref[thrm:main_iiib]{2.iii.\ref*{thrm:main_iiib}}.
\end{proof}

\section{Consequences of Theorem \ref{thrm:main}}
\label{sec:consequences}

We can use our characterization of threshold matroids to give an explicit enumeration of the number of isomorphism classes of threshold matroids on a ground set of size $n$. To do so, we first observe that a non-empty subset $T\subseteq [n]$ determines a unique shifted matroid isomorphism class.

\begin{lemma}
  \label{lem:iso}
Each shifted matroid isomorphism class contains a unique shifted matroid $M=\langle T \rangle$ whose bases are the elements of the principal order ideal generated by $T\in \binom{[n]}{k}$ in the component-wise partial order.
\end{lemma}
\begin{proof}
  By definition, every shifted matroid isomorphism class contains at least one representative $M=\langle T \rangle$. Thus, we need to show that, for two different size $k$ subsets $T,T'\in \binom{[n]}{k}$, $M=\langle T\rangle $ and $M'=\langle T'\rangle$ are not isomorphic. Let $T=T_1\cdots T_\ell$ and $T'=T_1'\cdots T_{\ell'}'$ be the block decompositions of $T$ and $T'$. Let $i$ be the smallest index such that $T_{\ell-i}\neq T'_{\ell'-i}$. Then either $|T_{\ell-i}|\neq |T'_{\ell'-i}|$ or the last element $t$ of $T_{\ell-i}$ is not equal to the last element $t'$ of $T'_{\ell'-i}$. Assume that either $|T_{\ell-i}|< |T'_{\ell'-i}|$ or $t<t'$. In the first case, Lemma \ref{lem:circuits} implies that $M$ has a circuit of size $|T_{\ell-i}|+|T_{\ell-i+1}|+\ldots+ |T_{\ell}|+1$ and $M'$ does not. Therefore $M\not\simeq M'$. In the second case, Lemma \ref{lem:circuits} implies that $M$ has $t'-t$ more circuits of size $|T_{\ell-i}|+\ldots + |T_\ell|+1 = |T_{\ell'-i}'|+\ldots + |T_{\ell'}'|+1$. Therefore $M\not\simeq M'$.
\end{proof}

\begin{repintrotheorem}{thrm:enumeration}
Let $F(n)$ be the number of isomorphism classes of threshold matroids on a ground set of size $n$. Then
\[F(n) = \binom{n+1}{2} + \binom{n+1}{4}+ \binom{n+1}{6} + \binom{n-1}{6}\]
where we take $\binom{n}{k}=0$ whenever $n<k$.
\end{repintrotheorem}
\begin{proof}
  By Lemma \ref{lem:iso}, an isomorphism class of a threshold matroid $M=\langle T \rangle $ is determined by a choice of subset $T\subseteq [n]$ with block structure as determined by Theorem \ref{thrm:main}. We break into cases and enumerate all subsets of $[n]$ which follow the block structure determined by Theorem \ref{thrm:main}. As we allow $M$ to have coloops and $T$ to start with a $1$, there are quite a few cases to enumerate. A subset $T\subseteq [n]$ determines a threshold matroid if and only if it satisfies one of the following mutually exclusive cases:
  \begin{enumerate}
  \item $T$ has one block.
  \item $T$ has two blocks.
  \item $T$ has three blocks and
 \begin{enumerate}
 \item $1\in T$.
 \item \label{case:3b} $1\not\in T$ and the second block of $T$ has size one.
 \item \label{case:3c}$1\not\in T$, the second block of $T$ has at least two elements and the second gap of $T$ has size one.
 \end{enumerate} 
 \item $T$ has four blocks, $1\in T$ and
  \begin{enumerate}
  \item \label{case:4a}the third block of $T$ has size one.
  \item \label{case:4b}the third block of $T$ has at least two elements and the second gap of $T$ has size one.
  \end{enumerate}
 \end{enumerate}

 By picking the starts and ends of the blocks, there are $\binom{n+1}{2}$ and $\binom{n+1}{4}$ subsets of $[n]$ with one and two blocks, respectively. Similarly, by choosing the end of the block containing $1$ and the starts and ends of the other blocks, there are $\binom{n}{5}$ subsets of $[n]$ which have three blocks and contain $1$.

 We can count the other cases with similar, slightly more intricate, counting arguments. We give such an argument for Case \ref{case:3b} and leave the others to the reader. Given a subset $T\subseteq [n]$ obeying the conditions of Case \ref{case:3b}, we construct a five element subset $S=\{s_{1},s_{2},s_{3},s_{4},s_{5}\}\subseteq \{2,3,\ldots,n+1\}$ determining $T$ such that $s_{3}+1<s_{4}$. Let $s_{1}$ be the start of the first block of $T$, $s_{2}$ be the start of the second gap of $T$, $s_{3}$ be the second block of $T$, $s_{4}$ be the start of the third block of $T$ and $s_{5}$ be the start of the fourth gap of $T$ (where $s_{5}=n+1$ if $T$ does not have a fourth gap). Because the third gap has size at least one, $s_{3}+1<s_{4}$. This process is reversible, as our recorded information completely determines $T$. As an example, the defining basis $T=3578\subseteq [8]$ corresponds to the set $S=34579\subseteq \{2,3,\ldots,9\}$. As there are $\binom{n-1}{5}$ five element subsets of $\{2,3,\ldots,n+1\}$ with $s_3+1<s_4$, there are also $\binom{n-1}{5}$ subsets of $[n]$ with three blocks, not containing $1$ and whose second block has size one.

 There are $\binom{n-2}{5}$, $\binom{n-1}{6}$ and $\binom{n-2}{6}$ subsets of $[n]$ satisfying Cases \ref{case:3c}, \ref{case:4a} and \ref{case:4b}, respectively. Therefore
 \begin{align*}
 \label{eq:2}
   F(n) &= \binom{n+1}{2}+ \binom{n+1}{4} + \binom{n}{5}+ \binom{n-1}{5} + \binom{n-2}{5} + \binom{n-1}{6} + \binom{n-2}{6}\\
          &= \binom{n+1}{2} + \binom{n+1}{4} + \binom{n+1}{6} + \binom{n-1}{6}
 \end{align*}
 
\end{proof}  

With this formula for the number of isomorphism classes of threshold matroids, a proof of Theorem \ref{thrm:almost} quickly follows.

\begin{repintrotheorem}{thrm:almost}
  Almost all shifted matroids are \underline{not} threshold. That is,
  \[\lim_{n\to\infty}\frac{\text{$\#$ of isomorphism classes of threshold matroids on a ground set of size $n$}}{\text{$\#$ of isomorphism classes of shifted matroids on a ground set of size $n$}}=0. \]
\end{repintrotheorem}
\begin{proof}
  By Lemma \ref{lem:iso}, the number of isomorphism classes of shifted matroids on a ground set of size $n$ is equal to $2^n-1$. On the other hand, by Theorem \ref{thrm:enumeration}, the number of isomorphism classes of shifted matroids on a ground set of size $n$ is equal to a sum of binomial coefficients. As binomial coefficients grow polynomially in $n$, our claim follows.
\end{proof}

We can also use our main theorem to give alternative proofs of the main results of Section $5$ of Deza and Onn's paper \cite{DO}. Their results are summarized in the following theorem. See Table \ref{tab:table} for a translation between the terminology used in \cite{DO} and the terminology used in this article. 

\begin{thrm}
\label{thrm:DO}
Let $M$ be a shifted matroid.
\begin{itemize}
\item (Theorem $5.4$ of \cite{DO}) If $M$ is rank three, then $M$ is threshold.
\item (Theorem $5.1$ of \cite{DO}) If $M$ is paving, then $M$ is threshold.
\item (Theorem $5.5$ of \cite{DO}) If $M$ is binary, then $M$ is threshold.
\end{itemize}
\end{thrm}

We start with the simplest theorem to reprove, Theorem $5.4$ of \cite{DO}. Using Theorem \ref{thrm:main}, we are able to greatly reduce the length and conceptual complexity of the proof.
\begin{prop}[Theorem $5.4$ of \cite{DO}]
Let $M=\langle T \rangle $ be a rank three shifted matroid. Then $M$ is threshold.
\end{prop}
\begin{proof}
Since $M$ is rank three, $|T|=3$ and $T$ can have at most three blocks. If $T$ has exactly three blocks, then the second block of $T$ must have size one. In either case, Theorem \ref{thrm:main} implies that $M$ is threshold.
\end{proof}

It turns out that shifted paving matroids are a relatively small class of both shifted matroids and paving matroids. In the below lemma, we classify their block structure. We'll then use this classification to give a succinct proof of Theorem $5.1$.

\begin{lemma}
\label{lem:paving}
Let $M= \langle T \rangle$ be a shifted matroid on ground set $[n]$. Then $M$ is paving if and only if
\[T =  \ell\, m\, m+1 \cdots n \]
where $1\leq \ell < m \leq n$.
\end{lemma}
\begin{proof}
Recall that $M$ is paving if all circuits of $M$ have size at least the rank of $M$. By Lemma \ref{lem:circuits}, the smallest circuit of $M$ has size equal to $1+|T_\ell|$ where $T_\ell$ is the last block of $T$. This is at least the rank of $M$ exactly when $T$ has the claimed form.
\end{proof}

We now prove Theorem $5.1$ of \cite{DO} using both Theorem \ref{thrm:main} and Lemma \ref{lem:paving}. Once again, this proof turns out to be very conceptually simple. 

\begin{prop}[Theorem $5.4$ of \cite{DO}]
  \label{prop:easy}
Let $M=\langle T \rangle $ be a paving shifted matroid. Then $M$ is threshold.
\end{prop}
\begin{proof}
By Lemma \ref{lem:paving}, $T$ can have at most two blocks. Hence, by Theorem \ref{thrm:main}, $M$ is threshold.
\end{proof}

Our strategy for proving Theorem $5.5$ of \cite{DO} is identical to that of our proof of Proposition \ref{prop:easy}. We will first characterize the block structure of binary matroids in Lemma \ref{lem:binary}. We will then use this to deduce that all binary shifted matroids are threshold. For the sake of brevity, we omit some of the computations in the proof of Lemma \ref{lem:binary}.

\begin{lemma}
\label{lem:binary}
Let $M= \langle T \rangle$ be a shifted matroid on ground set $[n]$. Then $M$ is binary if and only if $M$ is Boolean or
\[T  = 1\,2\cdots \ell-1\, \hat \ell\, \ell+1\cdots k\, m  \]
where $1\leq \ell \leq k < m \leq n$.
\end{lemma}
\begin{proof}
 A characterization of binary matroids is that $M$ is binary if and only if the symmetric difference of every pair of circuits of $M$ contains a circuit; see \cite[Theorem $9.1.2$]{oxley}. The Boolean matroid has no circuits and therefore is binary. We now focus on the second case of our lemma. Using Lemma \ref{lem:circuits}, one can explicitly write down the circuits of $M=\langle T \rangle$. It can then be checked that the symmetric difference of pairs of these circuits always contains another circuit.

In the other direction, suppose $M=\langle T\rangle$ is not Boolean or of our specified form. Let $T=T_1T_2\cdots T_\ell$ be the block decomposition of $T$. Then there are two elements $v_1<v_2$ contained in the complement of $T$ and which are smaller than at least two elements of $T$. Let $T_{i_1}$ and $T_{i_2}$ be the first blocks of $T$ after $v_1$ and $v_2$, respectively. Using Lemma \ref{lem:circuits}, one can check that 
\begin{gather*}
C_1 = v_1 T_{i_1} T_{i_1+1} \cdots T_\ell\\
C_2 = v_2 T_{i_2} T_{i_2+1} \cdots T_\ell
\end{gather*}
are circuits of $M$. Finally, using that $v_2$ is smaller than at least two elements of $T$ and is the largest element of the symmetric difference of $C_1$ and $C_2$, one can check that the symmetric difference of $C_1$ and $C_2$ is independent. This computation will prove that $M$ is not binary.
\end{proof}
\begin{prop}[Theorem $5.4$ of \cite{DO}]
Let $M=\langle T \rangle $ be a binary shifted matroid. Then $M$ is threshold.
\end{prop}
\begin{proof}
By Lemma \ref{lem:binary}, after contracting all coloops of $M$, $T$ can have at most two blocks. By Theorem \ref{thrm:main}, this means that $M$ is threshold.
\end{proof}

We conclude by giving a polynomial-time algorithm to check if a matroid $M$ is threshold. Suppose that $M$ is a rank $k$ matroid on a ground set $E$ of size $n$ with $m$ bases and $M$ is presented by a list of its bases. In Section $8$ of \cite{DO}, they show that there exists an algorithm that terminates in $O(m^{2}k^{2})$ time to determine if $M$ is shifted. They also give an algorithm to determine if $M$ is threshold. However, as part of their algorithm, they solve a linear program with $\binom{n}{k}$ many variables. Solving this linear program takes time polynomial in $\binom{n}{k}$. When $n$ and $k$ are both variable, $\binom{n}{k}$ grows exponentially in $k$. Therefore, their algorithm is not polynomial in $n$, $k$ and $m$. For this reason, they deem their algorithm non-viable for practical computations. For a fixed choice of $n$ and $k$, $m$ can be as large as $\binom{n}{k}$. However, in practice, $m$ is often much smaller than $\binom{n}{k}$. Therefore, it is desirable to find an algorithm, polynomial in $n$, $k$, and $m$, to check if $M$ is threshold. A key feature of such an algorithm is that, once $m$ is fixed, the algorithm takes polynomial time in $n$ and $k$. This feature is not present in the algorithm of Deza and Onn \cite[Section $8$]{DO}. Theorem \ref{thrm:polynomial} exhibits an algorithm, polynomial in $n$, $m$ and $k$, to decide if $M$ is threshold. We remind the reader that Table \ref{tab:table} provides a translation between the terminology used in \cite{DO} and the terminology used in this article. 

\begin{thrm}
\label{thrm:polynomial}
Suppose $M$ is a rank $k$ matroid on a ground set $E$ of size $n$ with $m$ bases and $M$ is presented by its list of bases. There is an algorithm that terminates in $O(m^{2}k^{2}+n\log(n))$ time to decide if $M$ is a threshold matroid.
\end{thrm}

If $M$ were presented to us as a principal order ideal generated by some $T\in \binom{[n]}{k}$, then, by using Theorem \ref{thrm:main} and checking the block structure of $T$, we could easily check if $M$ is threshold in $O(k)$ time. Thus, the main content of Theorem \ref{thrm:polynomial} is the construction of an algorithm to find an isomorphism between $M$, an arbitrary shifted matroid, and $\tilde M = \left\langle T \right\rangle$, a principal order ideal in the component-wise partial order of $\binom{[n]}{k}$. Recall that Theorem \ref{thrm:vicinal} lets us find such an isomorphism, provided that we can compute the vicinal preorder of $M$ on $E$.

Thus, our algorithm will proceed as follows:
We first check if $M$ is shifted. If $M$ is indeed shifted, we then compute the vicinal preorder of $M$ on $E$ to determine an isomorphism from $M$ to $\tilde M=\langle T \rangle$. We conclude by checking the block structure of $T$ to determine if $M$ is threshold. We verify that this procedure takes polynomial time in the following proof.

\vspace{1em}

\noindent \emph{Proof of Theorem \ref{thrm:polynomial}.} As remarked, Deza and Onn \cite[Section $8$]{DO} provide an algorithm that checks if $M$ is shifted in $O(m^2k^2)$ time. Thus, we can now assume that $M$ is a shifted matroid. With this in mind, we will now compute the total vicinal ordering of $E$. We can calculate all of the sets $N(i)$ and $N[i]$ in polynomial time. To do so, iterate over every $B\in M$ and $i\in B$, then add $B\setminus i$ to the sets $N(i)$ and $N[j]$ whenever $j\in B$. As there are $m$ bases and $k$ elements in each basis, this process will take $O(mk)$ time. Sorting $E$ to determine a bijection between $M$ and a shifted matroid $\tilde M= \langle T \rangle$ on $[n]$ will take $O(n\log(n))$ time. To determine the defining basis $T$, we need to compare the bases of $\tilde M$ at every index. This will take $O(mk)$ time. Lastly, we can check the block structure of $T$ in $O(k)$ time to determine if $M$ is threshold. As all of these steps are polynomial, our algorithm is polynomial and has time complexity $O(m^{2}k^{2}+n\log(n))$.

\printbibliography
\end{document}